\documentclass[12pt]{amsart}
\usepackage{amscd,amssymb,amsmath,latexsym,graphicx,amsxtra}
\usepackage[normalem]{ulem}
\usepackage[all]{xy}
\usepackage{xcolor}
\usepackage{tikz}

\usepackage{babel}
\usepackage{amsmath,amsfonts,amssymb,amsthm,amscd}
\usepackage{ucs} 
\usepackage[utf8x]{inputenc} 
\usepackage[T1]{fontenc} 
\usepackage{enumerate}
\usepackage{textcomp,url}
\usepackage{eucal}
\usepackage{xcolor}
\usepackage{yfonts}
\usepackage[noBBpl]{mathpazo}

\theoremstyle{definition}

\numberwithin{equation}{section}
\makeatletter

\renewcommand{\p@enumii}{}

\def\@enum@{\list{\csname label\@enumctr\endcsname}%
           {\usecounter{\@enumctr}\def\makelabel##1{
\normalfont\ignorespaces\emph{{##1}~}}
\setlength{\labelsep}{3pt}
\setlength{\parsep}{0pt}
\setlength{\itemsep}{0pt}
\setlength{\leftmargin}{0pt}
\setlength{\labelwidth}{0pt}
\setlength{\listparindent}{\parindent}
\setlength{\itemsep}{0pt}
\setlength{\itemindent}{0pt}
\topsep=3D3pt plus 1pt minus 1 pt}}

\makeatother

\renewcommand{\epsilon}{\ensuremath{\varepsilon}}
\renewcommand{\phi}{\ensuremath{\varphi}}

\renewcommand{\to}{\ensuremath{\longrightarrow}}
\renewcommand{\mapsto}{\ensuremath{\longmapsto}}

\makeatletter
\def\@map#1#2[#3]{\mbox{$#1 \colon\thinspace #2 \to #3$}}
\def\map#1#2{\@ifnextchar [{\@map{#1}{#2}}{\@map{#1}{#2}[#2]}}
\makeatother

\newtheoremstyle{theoremm}{}{}{\itshape}{}{\scshape}{.}{ }{}

\theoremstyle{theoremm}
\newtheorem{thm}{Theorem}
\newtheorem{lem}[thm]{Lemma}
\newtheorem{prop}[thm]{Proposition}

\newtheoremstyle{remarkk}{}{}{}{}{\scshape}{.}{ }{}

\theoremstyle{remarkk}
\newtheorem{defn}[thm]{Definition}
\newtheorem{rem}[thm]{Remark}

\newtheorem{ex}[thm]{Example}

\newtheoremstyle{comment}{}{}{\bfseries}{}{\bfseries}{:}{ }{}
\theoremstyle{comment}

\begin{document}
\title{Fixed point free homeomorphisms and the $R_{\infty}$-property}

\author{Daciberg Gon\c calves}
\address{Departamento de Matem\'atica, IME-Universidade de S\~ao Paulo Rua do Mat\~ao 1010 CEP: 05508-090 S\~ao Paulo-SP, Brazil}
\email{dlgoncal@ime.usp.br}
\author{Peter Wong}
\address{Department of Mathematics, Bates College, Lewiston, ME 04240, U.S.A.} 
\email{pwong@bates.edu}

\begin{abstract} Let $f:M\to M$ be a diffeomorphism on a compact connected smooth manifold of dimension at least $5$. It is known that the Nielsen number $N(f)$ of $f$ vanishes if and only if $f$ is isotopic to a fixed point free map. For any $n\ge 5$, there exists a compact $n$-dimensional nilmanifold $M$ such that {\it every} self homeomorphism on $M$ is isotopic to a fixed point free map. 
The proof depends on the fact that nilmanifolds are of {\it Jiang-type} and the fundamental group $\pi_1(M)$ has the property $R_{\infty}$ for certain nilmanifolds $M$.  
In this paper we construct the first known examples (an infinite family) of non-aspherical closed manifolds whose fundamental groups have property $R_{\infty}$ and that are also of Jiang-type. In particular, every self homeomorphism on such a manifold is isotopic to be fixed point free. The main objective of this work is to construct non-aspherical manifolds whose fundamental groups have property $R_{\infty}$ and that are also of Jiang-type.
\end{abstract}

\date{\today}

\maketitle

\section{Introduction}
 A space $X$ is said to have the fixed point property if every selfmap $f:X\to X$ has a fixed point, i.e., $f(x)=x$ for some $x\in X$. For instance, if the Lefschetz number $L(f)$ is nonzero for every $f$, then it follows from the celebrated Lefschetz fixed point theorem that $X$ has the fixed point property. Conversely, $L(f)=0$ does not guarantee that $f$ can be deformed to be fixed point free. For compact manifolds of dimension at least three, a classical theorem of Wecken asserts that the vanishing of the more subtle homotopy invariant, namely the Nielsen number $N(f)$, implies that $f$ is homotopic to a fixed point free map. From the point of view of dynamics, one is interested in knowing if {\it every} self homeomorphism 
 $f: X\to X$ can be deformed to be fixed point free.

In \cite{GW2}, it was shown that for every $n\ge 4$, there exists an $n$-dimensional compact nilmanifold $X$ such that {\it every} self homeomorphism is homotopic to a fixed point free homeomorphism. If $n\ge 5$ then the homotopy can be made to be isotopy. This result is based upon two important concepts, namely the property $R_{\infty}$ and the concept of a {\it Jiang-type} space. These manifolds were the first examples of an infinite family of Jiang-type spaces whose fundamental groups satisfy the so-called property $R_{\infty}$.

Let $\varphi: G\to G$ be a group endomorphism. The Reidemeister number $R(\varphi)$ is the number of $\varphi$-twisted conjugacy classes of elements of $G$ under the equivalence relation $\alpha \sim \sigma \alpha \varphi(\sigma)^{-1}$. In generalizing the classical Burnside-Frobenius theorem, A. Fel'shtyn and R. Hill \cite{FH} conjectured that $R(\varphi)=\infty$ if $\varphi$ is injective and $G$ has exponential growth. This conjecture has been proven to be false in general (see \cite{GW1}).
The study of $R(\varphi)$ in fact originated in the Nielsen fixed point theory where twisted conjugation arises naturally. A group $G$ is said to have property $R_{\infty}$ \cite{TW} if $R(\varphi)=\infty$ for every automorphism 
$\varphi\in {\rm Aut}(G)$. For a selfmap $f:X\to X$, the Reidemeister number $R(f)$ of $f$ is simply $R(\varphi)$ where $\varphi$ is the induced homomorphism on the fundamental group.  For manifolds of dimension at least $3$, the Nielsen number $N(f)$ is a sharp lower bound for the minimal number of fixed points in the homotopy class of $f$ and $N(f)\le R(f)$.

The computation of $N(f)$ is notoriously difficult. In 1963, Jiang introduced the Jiang subgroup $J(X) \subset \pi_1(X)$ (also known as the first Gottlieb group) of  the fundamental group. When $J(X)=\pi_1(X)$ then for every selfmap $f:X\to X$,  (i) $N(f)=0$ if $L(f)=0$ or (ii) $N(f)=R(f)$ if $L(f)\ne 0$. We say that a space is of {\it Jiang-type} if (i) or (ii) holds for all maps. Examples of Jiang-type spaces include simply-connected spaces, $H$-spaces, orientable coset spaces $G/K$ of compact connected Lie groups $G$ by closed subgroups $K$, nilmanifolds, certain solvmanifolds, and $\mathcal C$-nilpotent spaces where $\mathcal C$ is the class of finite groups. If $X$ is a Jiang-type space and $\pi_1(X)$ has property $R_{\infty}$ then it forces $N(f)=0$ for every self homeomorphism $f$ because $N(f)$ is always finite. Furthermore, if, in addition, $X$ is a compact manifold with $\dim X\ge 3$ then every self-homeomorphism is homotopic to be fixed point free. (For more background on Nielsen fixed point theory and related topics, see \cite{Jiang} and \cite{Br}.)

While many finitely generated groups are known to have property $R_{\infty}$, the only known examples of manifolds $M$ that are of Jiang-type with $\pi_1(M)$ having property $R_{\infty}$ are certain nilmanifollds and certain solvmanifolds (of type $\mathcal NR$) which are aspherical. Recently, in her dissertation \cite{vB}, van den Bussche constructed certain solvmanifolds for which every self homeomorphism has zero Nielsen number, extending certain results of \cite{GW1} (see also \cite{PY}). More precisely, she considered a class of spaces called $N_0$-spaces, i.e., the family of compact connected manifolds $M$ such that every self homotopy equivalence $f:M\to M$ has $N(f)=0$. Evidently, if $X$ is a compact connected triangulable manifold of dimension at least $3$ and $X$ is also an $N_0$-space then using a result of Jezierski \cite{Je}, it is straightforward to show that {\it every} self homeomorphism $f:X\to X$ is homotopic to a map $g$ with no periodic points of {\it any given} period.

The main purpose of this paper is to construct non-aspherical $N_0$-spaces and also $N_0$-spaces that are of Jiang-type. The goal is to focus on constructing 
manifolds whose fundamental groups are finitely generated virtually cyclic and have property $R_{\infty}$. The paper is organized as follows. In section 2, we recall some known results on virtually cyclic groups with property $R_{\infty}$ and some background on $N_0$-spaces, Jiang-type spaces, and $R_{\infty}$-spaces. In section 3, we construct $N_0$-spaces that are mapping tori of certain Lens spaces. In section 4, we consider coset spaces $U(n)/K$ of unitary groups by finite subgroups that are known to be of Jiang-type. Then we construct for each $n\ge 5$, an $N_0$-space of the form $U(n)/K_n$. In section 5, we show that for any finite subgroup $K$ of $U(2)$, the fundamental group $\pi_1(U(2)/K)$ does not have property $R_{\infty}$. In fact, we show that the space $U(2)/K$ does not have the property $R_{\infty}$.

The authors thank the CIRM (Luminy) for its support and hospitality under the "Research in Residence" program, February 25 - March 1, 2024 during which this  research  was conducted.  
The first  author  was partially supported   by Projeto  Regular  Fapesp Nielsen theory for maps  between spherical 3-manifolds and  some  homogeneous spaces  no. 2023/16051-5 and the travel for the second author was supported by a grant from Bates College.

\section{$R_{\infty}$- and  $N_0$-spaces}

In this section, we discuss $N_0$-spaces and the relationships among Jiang-type spaces and the $R_{\infty}$-property for groups and for spaces. We also recall the 
 $R_{\infty}$-property for virtually cyclic groups from \cite{GoWo} .

Let $X$ be a compact connected space.
\begin{defn} We say that   $X$ is an $R_{\infty}$-space if $R(f)= \infty$ for every homotopy equivalence $f : X \to  X$.
\end{defn}
Of course, if $\pi_1(X)$ has the $R_{\infty}$-property then $X$ is an $R_{\infty}$-space.
On the other hand, we will show later in this section that the
converse does not hold. From \cite{vB}  we have:
\begin{defn}   We say that $X$ is an $N_0$-space if $N(f)=0$ for
every homotopy equivalence
$f : X \to  X$.
\end{defn}
 For instance, $X = \mathbb RP^3\# \mathbb RP^3$ is an $R_{\infty}$-space but not an
$N_0$-space (see \cite[Theorem 7.3]{GWZ}).

Recall that a space $M$ is of {\it Jiang-type} if  for every map $f:M\to M$, either $L(f)=0 \Rightarrow N(f)=0$ or $L(f)\ne 0 \Rightarrow N(f)=R(f)$. Known Jiang-type spaces include $1$-connected spaces, $H$-spaces, generalized lens spaces, orientable $G/K$ of compact connected Lie group $G$ by closed subgroup $K$, nilmanifolds, $\mathcal C$-nilpotent spaces where $\mathcal C$ is the class of finite groups.

Evidently, if $M$ is of Jiang-type and is an $R_{\infty}$-space then $M$ is an $N_0$-space.

\subsection{Virtually Cyclic $R_{\infty}$-Groups}

Finitely generated non-elementary word hyperbolic groups have the $R_{\infty}$-property \cite{LL}. Non-elementary means not virtually cyclic so which (infinite) virtually cyclic groups have the $R_{\infty}$-property?

Let $G$ be a finitely generated virtually cyclic group so that

\begin{equation}\label{VC}
0 \to \mathbb Z \to G \to F \to 1
\end{equation}
is short exact with $F$ a finite group. Then $G$ falls into one of the following two types. Type I: if \eqref{VC} is not central; Type II if \eqref{VC} is central. If $G$ is of Type I then $G$ has property $R_{\infty}$ \cite[Prop. 2.8]{GoWo}.

Now suppose that $G$ is of Type II. Then $G\cong K\rtimes_{\theta} \mathbb Z$ for some finite group $K$ with action $\theta : \mathbb Z \to {\rm Aut}(K)$.
Necessary and sufficient conditions were given in \cite[Proposition 2.5]{GoWo} for $G$ to have property $R_{\infty}$. However, the statement is not entirely true and we take the opportunity to correct that. We thank Pieter Senden who pointed out  this error.

Proposition 2.5 of \cite{GoWo} states:
{\em Let $\mathbb Z = \langle t \rangle$, $H$ a finite group, $\theta(t)\in   {\rm Aut}(H)$, and $G = H\rtimes_{\theta} \mathbb Z$. Then $G$ has the $R_{\infty}$-property if, and only if, $\theta(t)$ is not conjugate to its inverse in ${\rm Aut}(H)$.}

The implication  {\em if $G$ has the $R_{\infty}$-property then  $\theta(t)$ is not conjugate to its inverse in ${\rm Aut}(H)$} holds and the proof given in \cite{GoWo} is correct. However, the  proof of the converse has a gap. The correct statement should  be  as  follows.  

\begin{prop}
\label{type2} 
Let $\mathbb Z = \langle t \rangle$, $H$ a finite group, $\theta(t)\in   {\rm Aut}(H)$, and $G = H\rtimes_{\theta} \mathbb Z$. Then $G$ has the $R_{\infty}$-property if, and only if, $[\theta(t)]$ is not conjugate to its inverse in ${\rm Out}(H)$.
\end{prop}
Here, $[\varphi]$ is the image of an element $\varphi \in {\rm Aut}(G)$ 
under the projection homomorphism ${\rm Aut}(H) \to {\rm Out}(H)={\rm Aut}(H)/{\rm Inn}(H)$.

The mistake in the proof of \cite[Proposition 2.5]{GoWo} occurred  when we assumed that an automorphism $\Phi: G \to G$ is of the form  
 $\Phi(h,t^r) =(\phi(h),t^ {kr})$ for $h\in H$ and $t^r \in \mathbb Z$. In fact, an arbitrary  automorphism should be of the form  
 $\Phi(h,t^r) =(\phi(h)\lambda^r,t^ {kr})$ for an  arbitrary $\lambda\in H$. With this modification, it is straightforward  
 to adapt the proof of the converse  given in the proof of  \cite[Propostion 2.5]{GoWo}. This modification will complete the proof of Proposition \ref{type2}  

\begin{rem}\label{abelian-aut}
It follows from Proposition \ref{type2} that if ${\rm Aut}(H)$ is abelian then $G$ has $R_{\infty}$ iff $\theta (1)$ is non-trivial and $\theta (1)$ is not of order $2$. 
\end{rem}

\begin{ex}\label{e1}
Let $G=\mathbb Z_5 \rtimes_{\theta} \mathbb Z, H=\mathbb Z_5$ so ${\rm Aut}(H)\cong \mathbb Z_4$ where $ \theta : \mathbb Z \to {\rm Aut}(H)$ is multiplication by $2$. Since $\theta(1)$ has order $4$, it follows that $G$ has property $R_{\infty}$. Furthermore, every automorphism $\varphi :G\to G$ induces the identity on $G/H=\mathbb Z$.
\end{ex}

\subsection{$R_{\infty}$-property for spaces  versus $R_{\infty}$-property
 for groups}

While it is the case that $\pi_1(M)$ has $R_{\infty}$ implies that $M$ is an $R_{\infty}$-space, the converse is false in general.
We now construct spaces  $M$  such   that  its  fundamental   group does
not  have  the $R_{\infty}$-property but the  space has
the  $R_{\infty}$-property,  i.e. for  any  homotopy equivalence, it  has
the  property  that the induced homomorphism  on  the fundamental
group     has infinite Reidemeister number.     

\subsubsection{Construction type I}

First, we give a general construction of such a space with prescribed property on the action of $\pi_1$ on $\pi_2$.

\begin{thm}\label{R-infty-space} Let  $F$  be a finite group  and $H=F\rtimes
_{\theta} \mathbb Z$  with  the  property that  any  automorphism of $H$ induces  the
identity  on  the quotient  $H/F\simeq \mathbb Z$.
There is a space   $X$  which  has   the  property that $\pi_1(X)\simeq \mathbb Z$,
$\pi_2(X)\simeq F$  and  the  action  of $\pi_1(X)$ on  $\pi_2(X)$  is
given  by  $\theta$. Furthermore,
this space has the $R_{\infty}$-property  but  not its  fundamental  group.
\end{thm}
\begin{proof}
Given a group  $G$ and a sequence of  $\mathbb Z[G]$-modules $A_i, i\geq 2$, it follows from \cite[Ex.3, p.460]{Sp} that there is a space  $X$  such  that
$\pi_1(X)\simeq  G$, $\pi_i(X)=A_i$ and the structure  of
$\mathbb Z[\pi_1(X)]$-module of $\pi_i(X)$ inherited from  the  space $X$ is
the given  $\mathbb Z[G]$-module  structure  of  $A_i$ for all $i\geq 2$.

Now let $X$ be such a space where $\pi_1(X)\simeq \mathbb Z$, $\pi_2(X)\simeq F$ such that 
the action of $\pi_1(X)$ on $\pi_2(X)$ is given by $\theta$.
Let  $f:X\to  X$  be  any
homotopy equivalence. The  action  of  $\pi_1$  on   $\pi_2$ is
natural   with  respect  to
maps so that $f_{\#}(\alpha\star
\beta))=f_{\#}(\alpha)\star f_{\#}(\beta)$ where $\alpha\in \pi_1(X), \beta \in \pi_2(X)$ and $\alpha \star \beta=(\theta(\alpha))(\beta)$. Applying  this
equality to the
homotopy equivalence $f$,  we  obtain two homomorphisms
$f_{\#1}$  and  $f_{\#2}$ which define an  automorphism of  the  group $H$.
By assumption, $f_{\#1}$ is  the  identity so $R(f)=R(f_{\#})\ge R(f_{\#1})=\infty$.  Since $\mathbb Z$ does not have property $R_{\infty}$, the  result  follows.
\end{proof}

\subsubsection{Construction type II}

Let $h:X\to X$ be a homeomorphism and $T_hX$ be the corresponding mapping torus. It follows that we have a fibration $X\to T_hX\to S^1$ over the circle and  that $G\simeq \pi_1(T_hX)$ has a presentation $G=\langle \pi_1(X), t\mid t\beta t^{-1}=h_{\#}(\beta), \beta \in \pi_1(X)\rangle$.

\begin{ex}\label{suspension}
Consider the Lens space $L(5,2)$. Using \cite{S}, there is a homeomorphism
   $f:L(5,2)\to L(5,2)$  which  induces an automorphism as multiplication by  $2$ on
$\pi_1(L(5,2))=\mathbb Z_5$ and  has  degree $-1$. Now  consider the homeomorphism $h=\Sigma f: X=\Sigma L(5,2)  \to
\Sigma L(5,2)$, the suspension of $f$. Observe  that this  space $X$ is simply connected and
$\pi_2(X)=\pi_2(\Sigma L(5,2))\simeq H_2(\Sigma L(5,2)) \simeq H_1(L(5,2)) \simeq \mathbb Z_5$. Now  if we  apply  the mapping
construction for the homeomorphism  $h=\Sigma f$, we obtain a space
$T_{h}X$ where $\pi_1(T_{h}X)\simeq \mathbb Z$, $\pi_2(T_{h}X)\simeq \mathbb Z_5$ and  the  action of
$\pi_1(T_{h}X)$ on  $\pi_2(T_{h}X)$  is  multiplication by
 $2$. From Example \ref{e1}, every automorphism of $\mathbb Z_5 \rtimes_{\theta} \mathbb Z$ induces the identity on the quotient $\mathbb Z$. It follows from Theorem 
 \ref{R-infty-space} that the space $T_hX$ has  the  $R_{\infty}$-property while $\pi_1(T_hX)\simeq \mathbb Z$ does not have.
\end{ex}

\section{Non-aspherical $N_0$-spaces, part I}

Let $X$  be a  compact connected manifold with $F=\pi_1(X)$ finite and $h:X\to X$ be a self homeomorphism. Denote by $M=T_{h}X$ the corresponding mapping torus. Then 
the natural projection  $M=T_{h}X\to S^1$ is  a fibre  bundle with fibre $X$.

Note that since $\pi_1(S^1)=\mathbb Z$ is free, we have $\pi_1(M)\cong F\rtimes \mathbb Z=\langle F, t \mid txt^{-1}=h_{\#}(x), x\in F\rangle$. Moreover, since $F=\pi_1(X)$ is finite, $F$ is characteristic in $\pi_1(M)$. If $\varphi: \pi_1(M)\to \pi_1(M)$ is an automorphism, we have the following commutative diagram
\begin{equation}\label{exact}
\begin{CD}
    1 @>>> \pi_1(X)    @>{i_{\#}}>>  \pi_1(M) @>{p_{\#}}>>    \mathbb Z @>>> 0 \\
    @.     @V{\varphi'}VV  @V{\varphi}VV   @V{\overline \varphi}VV                         @. \\
    1 @>>> \pi_1(X)    @>{i_{\#}}>>  \pi_1(M) @>{p_{\#}}>>    \mathbb Z @>>> 0 
 \end{CD}
\end{equation}
where the last vertical homomorphism is $\pm 1_{\mathbb Z}$. Again, since $F$ is finite, it follows from the diagram \eqref{exact} that $R(\varphi)=\infty$ if, and only if, $\overline \varphi=1_{\mathbb Z}$.

\begin{lem} \label{lens}
Suppose  $F\simeq \mathbb{Z}_p$ and $h_{\#}(l)=ql$ where $q^ 2\equiv -1 \  mod (p)$.  The group $\pi_1(M)$  has  the $R_{\infty}$--property  and so $\bar \varphi:\mathbb{Z}\to  \mathbb{Z}$ is the identity. 
\end{lem}
\begin{proof} The result follows from Proposition \ref{type2}.
\end{proof}

\begin{prop} \label{mappingtorus}
Suppose  $F\simeq \mathbb{Z}_p$ and $h_{\#}(l)=ql$ where $q^ 2\equiv -1 \  mod (p)$. Let $f':M\to M$ be a homotopy equivalence. Then $f'$ is homotopic to a fibre preserving map $f:M \to M$ so that we have the following commutative  diagram\\  
\begin{equation}
\begin{CD}
    X    @>{i}>>  M @>{p}>>    S^1 \\
    @V{ f|_X}VV  @V{ f}VV   @V{\overline f}VV                         @. \\
    X    @>{i}>>  M @>{p}>>    S^1
 \end{CD}
\end{equation}
which, in turn, induces the commutative diagram \eqref{exact} at the fundamental group level 
where $\varphi=f_\#$  and  $\varphi'=(f|_{X})_{\#}$ and $\overline{\varphi}=\overline f_\#=1_{\mathbb Z.}$
\end{prop}
\begin{proof}  Since  $f'$ is a homotopy equivalence, the map $p\circ f'$ is homotopic  to $p$ because  $\overline \varphi$ is the identity. We obtain the map $f$ using the Covering Homotopy Property of the fibration $M\to S^1$.
\end{proof}

The following result is immediate.
\begin{thm}\label{general-N_0}
Let $M$ and $f$ be as in Proposition \ref{mappingtorus}. The map $f$ can be deformed to a  fixed  point  free map and $M$ is an $N_0$-space.
\end{thm}

Now we search for groups of the form  $G=F\rtimes_{\theta} \mathbb Z$ such  that any automorphism 
induces the identity on the quotient $G/F$. Note that such $G$ is virtually cyclic and necessarily has property $R_{\infty}$.  Once such a group is  given, we look for  a homeomorphism $h :X\to X$ such that  $\theta(1)=h_{\#}$. Classical lens spaces are natural candidates for such a space $X$. First, we recall some known results regarding realizing a degree $\pm 1$ map by a homeomorphism on three dimensional lens spaces. 

\begin{thm}\label{SUN} (\cite{S} or \cite[Propositions 1.3, 1.4]{PHZ})
(i) Suppose that $f$ is a degree $1$ self-map on $L(p,q)$. Then $f$ is homotopic to an orientation-preserving homeomorphism if and only if
$$f_{*}(l)=\pm l, \  if \  p \  does \ not \  divide \ q^2-1,  \ and  \  f_{*}(l)= \pm l, \ \pm ql,\ if \  p\  divides \ q^2-1,$$
for all $l \in \pi_1(L(p,q))$, where $f_{*}$  is the endomorphism of $\pi_1(L(p,q))$ induced by $f$.  \\
(ii) $L(p,q)$  admits an orientation-reversing homeomorphism if and only if
$$q^2\equiv -1\ \  (mod p).$$\\
In this case, a degree  $-1$  self-map $f$ on $L(p,q)$ is homotopic to an orientation-reversing
homeomorphism if and only if
$$f_{*}(l) = \pm ql, \ for \ all \ l \in \pi_1(L(p,q)),$$
where $f_{*}$ is the endomorphism of $\pi_1(L(p,q))$ induced by $f$.\\
\end{thm}

\subsection{Examples - in dimension $4$}

We now describe how to use Lemma \ref{lens} to construct an $N_0$-space using the procedure above for a pair $(L(p,q), \theta)$ where $L(p,q)$ is a lens space and $\theta$ is an automorphism of  the fundamental group  $\mathbb Z_p$  of the  lens spaces $L(p,q)$. We also write $(L(p,q), k)$ when the automorphism $\theta$ is multiplication by $k$.

\begin{thm}\label{lens-space-3D}
Let $p$ be a positive integer of the form either  $p_1^{e_1}$ or $2p_1^{e_1}$ where $p_1$ is an odd prime of the form $4k+1$. (a) There exists a positive integer $q$ such that $q^2\equiv -1$  \mbox{mod}  $p$. (b) The mapping torus $T_hX$ is an $N_0$-space where $X$ and $h$ are given by the pair $(L(p,q),q)$.
\end{thm}
\begin{proof}
(a)  When $p=p_1^{e_1}$  for $p_1$ a prime of the form  $4k+1$, the existence of $q$ follows from  \cite[page 113 lines 4-8]{GG}.  
In  the case of $p=2p_1^{e_1}$, since  ${\rm Aut}(\mathbb Z_{2p_1^{e_1}})={\rm Aut}(\mathbb Z_{p_1^{e_1}})$, the proof of   the result  is similar. 

(b) Using the integer $q$ obtained from (a), consider the lens space $X=L(p,q)$. By Theorem \ref{SUN}(ii) (or \cite[Proposition 1.2 item ii)]{S}), there is a homeomorphism $h:X\to X$ of degree $-1$ such that $h_{\#}=\theta$ where $\theta$ is multiplication by $q$.  Thus, Lemma \ref{lens} applies and the result follows from Theorem \ref{general-N_0}.
\end{proof}

\begin{rem} We do not consider the lens spaces  given  by   \cite[Proposition 1.2 item i)]{S} since for such lens spaces there are no maps  of  degree  $-1$. For the other spherical $3$-manifolds (except $S^3$, \cite[p. 869]{S}), there are no degree $-1$ maps. By Remark \ref{abelian-aut}, we do not consider degree $1$ maps on lens spaces. Similarly, by Proposition \ref{type2}, we also do not consider the spherical $3$-manifolds with fundamental group isomorphic to $T^*_{24}, O^*_{48}$ or $I^*_{120}$ for they have 
${\rm Out} \cong \mathbb Z_2$.
\end{rem}

Theorem \ref{lens-space-3D} provides an infinite family (each for every odd prime of the form $4k+1$ and there are an infinite number of such primes (see e.g. \cite[Theorem 6.1, p.105]{T})) of four dimensional compact connected manifolds for which every self homotopy equivalence is homotopic to be fixed point free. The simplest example for which this theorem applies is  $(L(5,2),2)$.

\subsection{Examples - higher dimension}

Now we consider examples of lens spaces in higher dimensions using \cite{PHZ} which generalizes the results of \cite{S} in the case of lens spaces.

Recall from \cite{PHZ} the following result which generalizes Theorem \ref{SUN}(i).

\begin{thm}\label{PHZ1}
 Assume that $f$ is a degree 1 self-map on $L(p;q_1,q_2,...,q_n)$. Then, $f$ is homotopic to an orientation-preserving homeomorphism if and only if
$$f_{*}(l) = ql,\ \  for \ \   all \ \  l\in  \pi_1(L(p;q_1,q_2,...,q_n)),$$
where $f_{*}$ is the endomorphism on $\pi_1(L(p;q_1,q_2,...,q_n))$ induced by $f$, $q \in \mathbb Z_p$ is \\
\noindent coprime to $p$ with a permutation $\sigma$, such that 
$$q^n\equiv 1 \ \  (\mbox{mod} p), \ \ \  q_i \equiv\pm qq_{\sigma(i)} \ \  (\mbox{mod} p)$$
for all $i = 1,2,...,n.$
\end{thm}

Unlike in dimension $3$ where we did not consider degree $1$ maps because of Remark \ref{abelian-aut}, we can construct many higher dimensional lens spaces for which Lemma \ref{lens} applies.  

For  simplicity we will restrict to the  case  where $p$ is  prime.        Given   $n\ge 3$, choose a prime  $p\ge 5$ such  that ${\rm Aut}(\mathbb Z_p)$ admits a non-trivial  automorphism 
$\theta$  (given by multiplication by some  $q$, i.e., $\theta(x)=qx$.) such that $\theta^2\ne 1$ and $\theta^n=1$. Then the orbit of   $q_1=1$ under $\theta$ gives rise to $q_i, i=1,..., n/d$ and a permutation $\sigma$ such that $q_i \equiv\pm qq_{\sigma(i)} \ \  (\mbox{mod} p)$ where $d$ is  the order of $q$.

For instance, take  $p=5$ and $q=2$.  Observe that  $q^4\equiv 1$  \mbox{mod} $5$. 
Based on  this,  define  $q_1=3$, $q_2=4$,  $q_3=2$, $q_4=1$ with the permutation $\sigma=(1342)$ with $q_i=i$.  The  map  which  realizes this has  degree $1$. By Theorem \ref{PHZ1}, this map  can  be realized 
by a homeomorphism $h$ with $h_{\#}$ be given by multiplication by $q=2$. Note that Lemma \ref{lens} is still valid if we replace the condition on $q^2$ by $q^2\not\equiv 1$ \mbox{mod} $p$.  Now, we can apply Theorem \ref{general-N_0} to obtain an $N_0$-space $T_hX$. Moreover, since $\dim T_hX\ge 5$, it follows that every self homeomorphism of $T_hX$ is isotopic to be fixed point free.

Similarly, we can take $n=3$,  $p=7$ and $q=4$.  Observe that  $q^3\equiv 1$  \mbox{mod} $7$. 
Define  $q_1=2$, $q_2=4$, $q_3=1$ with the permutation $\sigma=(132)$.  The  map  which  realizes this has  degree $1$. By Theorem \ref{PHZ1}, this map  can  be realized 
by a homeomorphism $h$ with $h_{\#}$ be given by multiplication by $q=4$. Now, applying Theorem \ref{general-N_0} yields an $N_0$-space $T_hX$ and since $\dim T_hX=6\ge 5$, it follows that every self homeomorphism of $T_hX$ is isotopic to be fixed point free. Using the same argument, we can take $n=6$,  $p=7$ and $q=4$.  Since $q^3\equiv 1$  \mbox{mod} $7$, we also have $q^6\equiv 1$  \mbox{mod} $7$. 
Define  $q_1=2$, $q_2=4$, $q_3=1$,  $q_4=6$, $q_5=5$, $q_6=3$ with the permutation $\sigma=(132)(465)$.  Then we can obtain  an $N_0$-space $T_hX$ of dimension $(2(6)-1)+1=12$ such that every self homeomorphism of $T_hX$ is isotopic to be fixed point free.\\

Note that for any prime $p\ge 5$ and $q=2$, we can always find a permutation $\sigma$ (not necessarily a $(p-1)$-cycle) on $\{1,...,p-1\}$ and $\{q_i\}=\{1,...,p-1\}$ so that the hypotheses in Theorem \ref{PHZ1} are satisfied and that $q^2\not\equiv 1$ \mbox{mod} $p$.

In fact, we have the following.      

\begin{thm}\label{higher-lens1}
Let $p\ge 5$ be a prime, $X=L(p; 1,..., p-1)$ be a $(2(p-1)-1)$-dimensional lens space, and $\theta : \mathbb Z_p \to \mathbb Z_p$ be the automorphism given by multiplication by $q=2$. Then $M=T_hX$ is an $N_0$-space where $h:X\to X$ is a degree $1$ homeomorphism with $h_{\#}=\theta$. Moreover, every self homeomorphism of $M$ is isotopic to be fixed point free.\end{thm}

The next result from \cite{PHZ} generalizes Theorem \ref{SUN}(ii).

\begin{thm}\label{PHZ-1}
A lens space $L(p;q_1,q_2,...,q_n)$ admits an orientation-reversing homeomorphism if and only if there exist 
some $q \in \mathbb Z_p$ coprime to $p$ with a permutation $\sigma$, such that
$$q^n=-1 \ \ (\mbox{mod} p),  \ \ \  q_i \equiv\pm qq_{\sigma(i)} \ \  (\mbox{mod} p),$$
for all $i = 1,2,...,n$. In this case, a degree $-1$ self-map $f$ on $L(p;q_1,q_2,...,q_n)$
homotopic to an orientation-reversing homeomorphism if and only if 
$$f_{*}(l)=ql,   \ \ for \ \ all \ \  l\in \pi_1(L(p;q_1,q_2,...,q_n)),$$
where the $f_{*}$ is the endomorphism of $\pi_1(L(p;q_1,q_2,...,q_n))$ induced by $f$,  $q\in  \mathbb Z_p$ is coprime to $p$ with a permutation
 $\sigma$,  such that 
 $$q^n\equiv -1 \ \  (\mbox{mod} p), \ \ \ \ \ \ \ \ \  q_i\equiv\pm qq_{\sigma(i)} \ \  (\mbox{mod} p),$$
  for all $i = 1,2,...,n.$.
\end{thm}

Based on this result, one can construct an $N_0$-space as follows.  Take  $p=11$, $q=2$ and  $n=5$.  Observe 
that  $2^5\equiv -1$  mod(11).  Now  we  define    $q_1=5$, $q_2=8$,  $q_3=4$, $q_4=2$, $q_5=1$.  Observe  that, modulo $11$, 
$$q_1\equiv 2\cdot q_2; \quad q_2\equiv 2\cdot q_3;\quad q_3\equiv 2\cdot q_4;\quad q_4\equiv 2\cdot q_5;\quad q_5\equiv -2\cdot q_1,$$
so the permutation is $\sigma=(12345)$ . 
It follows from Theorem \ref{PHZ-1} that we have a homeomorphism $h$ of  $X=L(11;5,8,4,2,1)$ of  degree  $-1$. Thus, the mapping torus $T_hX$ is an $N_0$-space.
   
It is straightforward to extend this to the following by repeating the sequence $q_1,..., q_5$, i.e., $\{q_i\}=\{q_{j+5k}\mid 1\le j\le 5\}$.

\begin{thm}\label{higher-lens-1}
For any positive integer $N$, let $q_1=5; q_2=8; q_3=4; q_4=2; q_5=1$ and $q_{i+5j}=q_i$ for $1\le i\le 5$ and $1\le j\le N-1$. There is an $(2(5N)-1)$-dimensional $N_0$-space $M=T_hX$ which is the mapping torus of the lens space $X=L(11;q_1,..., q_{5N})$ where the homeomorphism $h:X\to X$ has degree $-1$ and $h_{\#}$ is multiplication by $2$. Moreover, every self homeomorphism of $M$ is isotopic to be fixed point free.
\end{thm}

While we have constructed many $N_0$-spaces as in Theorem \ref{higher-lens1} and in Theorem \ref{higher-lens-1}, we do not know if any of these spaces can be of Jiang-type. Next, we focus on certain Jiang-type spaces with virtually cyclic fundamental groups.

\section{Non-aspherical (Jiang-type) $N_0$-spaces, part II}

Let $G$ be a compact connected Lie group; $K$ a finite subgroup. It follows from \cite{W} (see also \cite{GW3}) that $M=G/K$ is of Jiang-type. Moreover, it was shown in \cite{GW3} that
$$
0\to \pi_1(G) \to \pi_1(G/K) \to K \to 1
$$
is a central extension. In particular, when $G=U(n)$, $\pi_1(U(n)/K)$ is a virtually cyclic group of Type II so that Proposition \ref{type2} is applicable.

Given a finite subgroup $K\subset U(n)$ we begin by describing the fundamental group of the homogeneous space $U(n)/K$ (left cosets). The determinant map $\det:U(n)\to S^1$  provides a fibration 
$SU(n)\to U(n)\stackrel{\det} \to S^1$ where the fibre, the pre-image of  $1\in S^1$, is the subgroup  $SU(n)$. Observe that $\det(K)$ is a finite subgroup of $S^1$, therefore it is a finite cyclic group. 
By considering the coset spaces, we obtain  a fibration
$$SU(n)/K' \to U(n)/K \to  S^1/\det(K)$$
where $K'=K\cap SU(n)$ and  $S^1/\det(K)$ is homeomorphic to  $S^1$.   

The last $5$-terms of the long exact sequence in homotopy associated to the above fibration is:

$$ 1\to K' \to \pi_1(U(n)/K) \to \mathbb Z\to 1.$$
\noindent Further, this short exact sequence splits (since $\mathbb Z$ is free) and the action of a generator of  $\mathbb Z$ 
is obtained as follows:\\
 Choose a matrix  $A\in K$ such that $\det(A)$ is a generator of $\det(K)$ and let $\alpha \in K'$. The action of 
 the generator of $\mathbb Z=\pi_1(S^1/\det(K))$  defined by $A$ on $\alpha$ is given the conjugation
  $A\circ \alpha\circ A^{-1}$. In order to decide  if the group $K'\rtimes_{\theta} \mathbb Z$ has or not the $R_{\infty}$-property 
  we use Propostion \ref{type2}.

\subsection{Jiang-type $N_0$-spaces}

Let $n=5$. Consider the elements $A,B\in U(5)$ given by

$$A=\begin{pmatrix}
               \omega & 0 & 0 & 0 & 0  \\
               0 & 0 & 0 & \omega & 0  \\
               0 & \omega & 0 & 0 & 0  \\
               0 & 0 & 0 & 0 & \omega \\
               0 & 0 & \omega & 0 & 0                                       
                                                         \end{pmatrix}, \qquad B=\begin{pmatrix}
               0 & 0 & 0 & 0 & 1\\
               1 & 0 & 0 & 0 & 0 \\
               0 & 1 & 0 & 0 & 0 \\
               0 & 0 & 1 & 0 & 0 \\
               0 & 0 & 0 & 1 & 0                                        
                                                         \end{pmatrix}$$
                                                         
where $\omega^{20}=1$.

Let $K\le U(5)$ be

$$K=\langle A,B \mid A^{20}=B^5=I, A^4B=BA^4, ABA^{-1}=B^2 \rangle.
$$

Consider the subgroup $K'=\langle A^4, B\rangle \cong \mathbb Z_5 \oplus \mathbb Z_5$. It is easy to see that 

$$1\to K' \to K \to \mathbb Z_4 \to 1$$
is a short exact sequence of finite groups and $\mathbb Z_4$ acts on $K'$ via $\begin{pmatrix}
                                                         1 & 0 \\
                                                         0 & 2
                                                         \end{pmatrix}$.

It follows that

$$
1\to K' \to \pi_1(U(5)/K) \to \mathbb Z \to 1
$$
is a short exact sequence such that $\pi_1(U(5)/K)\cong K' \rtimes_{\theta} \mathbb Z$
with action $\theta$ given by $\theta(1) =\begin{pmatrix}
                                                         1 & 0 \\
                                                         0 & 2
                                                         \end{pmatrix}$. Note that $\theta(-1)=\begin{pmatrix}
                                                         1 & 0 \\
                                                         0 & 3
                                                         \end{pmatrix}$ is the inverse of $\theta(1)$.

It is straighforward to check that for any $V\in GL_2(\mathbb Z_5)$,
$$
\begin{pmatrix}
                                                         1 & 0 \\
                                                         0 & 2
                                                         \end{pmatrix}V \ne V\begin{pmatrix}
                                                         1 & 0 \\
                                                         0 & 3
                                                         \end{pmatrix}.$$

Since for any $k\ge 1$, we can embed $M\in U(n)$ in $U(n+k)$ as $\begin{pmatrix}
                                                                                        M & 0 \\
                                                                                        0 & I_k \end{pmatrix}$ where $I_k$ is the $k \times k$ identity matrix, we have the following.

\begin{thm}\label{unitary} For any $n\ge 5$, there exists a finite group $K_n \le U(n)$ such that $M(n)=U(n)/K_n$ is a Jiang-type, $R_{\infty}$, $N_0$-space. Furthermore, every self homeomorphism of $M(n)$ is isotopic to  be fixed point free. Here $K_5=K$ as described above.
\end{thm}

\subsection{Another Example}

Now, we give another example of a finite subgroup $K\le U(7)$ with the desired property but $K$ is not of the form $\begin{pmatrix}
                                                                                        													K_5 & 0 \\
                                                                                        													0 & I_2 \end{pmatrix}$ as in Theorem \ref{unitary}.
													
Consider the elements $A,B\in U(7)$ given by

$$A=\begin{pmatrix}
               \omega & 0 & 0 & 0 & 0 & 0 & 0 \\
               0 & 0 & 0 & 0 & \omega & 0 & 0 \\
               0 & \omega & 0 & 0 & 0 & 0 & 0 \\
               0 & 0 & 0 & 0 & 0 & \omega & 0 \\
               0 & 0 & \omega & 0 & 0 & 0 & 0 \\
               0 & 0 & 0 & 0 & 0 & 0 & \omega \\
               0 & 0 & 0 & \omega & 0 & 0 & 0                                        
                                                         \end{pmatrix}, \qquad B=\begin{pmatrix}
               0 & 0 & 0 & 0 & 0 & 0 & 1\\
               1 & 0 & 0 & 0 & 0 & 0 & 0 \\
               0 & 1 & 0 & 0 & 0 & 0 & 0 \\
               0 & 0 & 1 & 0 & 0 & 0 & 0 \\
               0 & 0 & 0 & 1 & 0 & 0 & 0 \\
               0 & 0 & 0 & 0 & 1 & 0 & 0 \\
               0 & 0 & 0 & 0 & 0 & 1 & 0                                        
                                                         \end{pmatrix}$$
                                                         
where $\omega^{21}=1$.

Let $K\le U(7)$ be

$$K=\langle A,B \mid A^{21}=B^7=I, A^3B=BA^3, ABA^{-1}=B^2 \rangle.
$$

Consider the subgroup $K'=\langle A^3, B\rangle \cong \mathbb Z_7 \oplus \mathbb Z_7$. It is easy to see that 

$$1\to K' \to K \to \mathbb Z_3 \to 1$$
is a short exact sequence of finite groups and $\mathbb Z_3$ acts on $K'$ via $\begin{pmatrix}
                                                         1 & 0 \\
                                                         0 & 2
                                                         \end{pmatrix}$.

It follows that

$$
1\to K' \to \pi_1(U(7)/K) \to \mathbb Z \to 1
$$
is a short exact sequence such that $\pi_1(U(7)/K)\cong K' \rtimes_{\theta} \mathbb Z$
with action $\theta$ given by $\theta(1) =\begin{pmatrix}
                                                         1 & 0 \\
                                                         0 & 2
                                                         \end{pmatrix}$. Note that $[\theta(1)]^2=\begin{pmatrix}
                                                         1 & 0 \\
                                                         0 & 4
                                                         \end{pmatrix}$ is the inverse of $\theta(1)$.

We can check that there exists NO $V\in GL_2(\mathbb Z_7)$ such that 
$$
\begin{pmatrix}
                                                         1 & 0 \\
                                                         0 & 2
                                                         \end{pmatrix}V=V\begin{pmatrix}
                                                         1 & 0 \\
                                                         0 & 4
                                                         \end{pmatrix}.$$

\section{Finite  subgroups  of  $U(2)$}

We have constructed for each $n\ge 5$ a finite subgroup $K_n$ such that $\pi_1(M(n))$ has the  $R_{\infty}$-property where $M(n)=U(n)/K_n$. It is therefore natural to ask whether such examples can be found when $n=2,3,$ or $4$. In this section we show that for any finite subgroup $K\subset U(2)$, the fundamental group of  the homogeneous space $U(2)/K$ does not have the $R_{\infty}$-property. Moreover, we show that the space $U(2)/K$ is not an $R_{\infty}$-space.

\subsection{$\pi_1(U(2)/K)$ and the $R_{\infty}$-property} 

First, we briefly describe the classification by Du Val \cite{Du} and relate the subgroup $K'$ with the notations of finite subgroups used by Du Val.

For any $A\in U(2)$, the corresponding transformation $T_A$ sends $x \mapsto \epsilon x \kappa$ for some $\epsilon \in S^1$ (unit complex numbers) and $\kappa \in Sp(1)$ (unit quaternions), for every $x\in \mathbb C^2$. Equivalently,
\begin{equation}\label{DV}
A=\begin{pmatrix}
          \epsilon & 0 \\
          0 & \epsilon \end{pmatrix} \begin{pmatrix}
          						a & -\bar c \\
						c & \bar a \end{pmatrix} \end{equation}
where $a,c \in \mathbb C$, $a \bar a + c \bar c=1$, $\epsilon \in S^1$,  and $\kappa=a+cj \in Sp(1)$. Thus for every $A$, we can associate to it a pair $(\epsilon, \kappa) \in S^1 \times Sp(1)$. Note that the pair $(\epsilon, \kappa)$ is not unique since we can also associate to $A$  the pair
$(-\epsilon, -\kappa)$. In \cite{Du}, it was shown that any finite subgroup $K$ of $U(2)$ is determined by four groups as follows. Given $K$, let $R=\{r \mid \exists \ell \in S^1, (\ell, r)\in K\}$, $R_K=\{r \mid (1,r)\in K\}$. Similarly, let $L=\{\ell \mid \exists r\in Sp(1), (\ell, r)\in K\}$ and $L_K=\{\ell \mid (\ell, 1)\in K\}$. Moreover, $R/R_K \cong L/L_K$. Note that different isomorphisms $\overline{\phi}: L/L_K \to R/R_K$ may yield different finite subgroups $K$. According to \cite{Du}, every finite subgroup is denoted by $(L/L_K, R/R_K)$ together with an isomorphism $\overline{\phi}$ between $L/L_K$ and $R/R_K$ and the identification $(\epsilon, \kappa)\sim (-\epsilon, -\kappa)$. In other words, every finite subgroup is given by 
\begin{equation}\label{finiteK}
K=\{(\epsilon, \kappa)\in L\times R \mid \overline{\phi}(\overline{\epsilon})=\overline{\kappa}\}/\sim.
\end{equation}

Now, given a finite subgroup $K \subset U(2)$, we recall that $K'$ is the kernel of the determinant map $\det : K \to S^1$. Suppose $A\in K$. It follows from \eqref{DV} that $A\in K'$ iff $\epsilon^2=1$ or $A$ corresponds to $(1,\kappa)$. In other words, $K'$ is the subgroup $\{(1,r) \in K\}$ and hence $K' \cong R_K$.

For the list  of  all  finite groups of  $U(2)$, we will use  the classification given   in \cite[Chapter 10, section 10.1 page 98]{Co1}  and \cite[Theorem 2.2]{FaPa}. From there  the finite groups of unitary transformations in the plane are:
\begin{itemize}
\item[(1)] $(\textswab{C}_{2m}/\textswab{C}_f;  \textswab{C}_{2n}/\textswab{C}_g)_{d},$\ \ \  \ \ \ \ \ \  \ \ \ \ \ \ \ \ \ \ \ \ \ \ of order $gm=fn$
\item[(2)] $\langle p,2,2\rangle_m=(\textswab{C}_{2m}/\textswab{C}_{2m};  {\bf D}_{p}/{\bf D}_p),$ \ \ \ \ \ \ \ \ \ \ \ \ \ \ \ $4mp$ 
\item[(3)]   $(\textswab{C}_{4m}/\textswab{C}_{2m};  \langle p,2,2\rangle= {\bf D}_{p}/\textswab{C}_{2p})
,$ \ \ \ \ \ \ \ \ \ \ \ \ \ \ \  \ \ \ \ \ \ \ \ \ \ \ \ \ $4mp$ 

\item[(3')]   $(\textswab{C}_{4m}/\textswab{C}_{m};  \langle p,2,2\rangle={\bf D}_{p}/\textswab{C}_{p}), \ m \ and \ \ p\ \ odd,
$   \ \ \ \ \ $2mp$

\item[(4)]   $(\textswab{C}_{4m}/\textswab{C}_{2m};  \langle 2p,2,2\rangle={\bf D}_{2p}/ \langle p,2,2\rangle= {\bf D}_{p}),$ \ \ \ \ \ \  \ \ \ \ \ \ \ \ \ \ \ \ \ $8mp$ 

\item[(5)]   $ \langle 3,3,2\rangle_m=(\textswab{C}_{2m}/\textswab{C}_{2m};  {\bf T}/{\bf  T}),$   \ \ \ \ \ \ \ \  \ \ \ \ \ \ \ \ \ \ \ \ \ $24m$ 

\item[(6)]   $(\textswab{C}_{6m}/\textswab{C}_{2m};  \langle 3,3,2\rangle/\langle 2,2,2\rangle={\bf T}/{\bf Q}_8),$    \ \ \ \ \ \ \ \ \  \ \ \ \ \ \ \ \ \ \ \ \ \ $24m$ 

\item[(7)]   $ \langle 4,3,2 \rangle_m=(\textswab{C}_{2m}/\textswab{C}_{2m};  {\bf O}/{\bf O}),$   \ \    \ \ \ \ \  \ \ \ \ \ \ \ \ \ \ \ \ \ $48m$ 

\item[(8)]   $(\textswab{C}_{4m}/\textswab{C}_{2m};  \langle 4,3,2\rangle/\langle 3,3,2 \rangle={\bf O}/{\bf T}),$   \ \ \ \ \ \ \ \ \  \ \ \ \ \ \ \ \ \ \ \ \ \ $48m$ 

\item[(9)]   $ \langle 5,3,2 \rangle_m=(\textswab{C}_{2m}/\textswab{C}_{2m};  {\bf I}/{\bf  I}),$  \  \ \ \ \ \ \ \ \ \ \  \ \ \ \ \ \ \ \ \ \ \ \ \ $120m$ 

\end{itemize}
For the case (1)    the numbers $f,g,m,n,d$ are  positive integers such that $f\equiv g(mod(2)$, $gm=nf$, $f$ divides $2m$, $d$ is relatively prime to $2m/f$, and $1\leq d<m/f$. Then there is a  group $(\textswab{C}_{2m}/\textswab{C}_f;  \textswab{C}_{2n}/\textswab{C}_g)_{d},$ consisting  of the $gm=fn$ transformations 
\begin{equation}\label{abelian-K'}
e^{\mu \pi i/m}xe^{d\nu \pi i/n}
\end{equation}
where $\mu=0,1,\cdots, \lambda m-1$; $\nu=0,1,\cdots,n-1$; $\mu\equiv \nu(mod \ \ \lambda m/f)$, and $\lambda=$ 1 or 2 according as $f$ and $g$ are odd or even (see \cite[page 90 subsection 9.2]{Co1}.

The group denoted by $\langle p,q,r\rangle_{m}$ has the following presentation (see \cite[page 90 subsection 9.2 eq.9.32]{Co1}).
:\\
$$\langle  A, B, C, Z \ | \  
A^p=B^q=C^r=ABC=Z^m; \ \ [Z,A]=[Z,,B]=[Z,C]=1 \rangle.  $$
When $m=1$ we denote simply by   $\langle p,q,r\rangle$.  We  have the following identifications: $\langle p,2, 2\rangle={\bf D}_p; \langle 2p,2,2\rangle={\bf D}_{2p} \ (dicyclic \  groups);  \langle 2,2,2\rangle={\bf Q}_8$ \ ($quaternionic \ group \ of \ order \  eight$); $\langle 3,3,2\rangle={\bf T};  \langle 4,3,2\rangle={\bf O};  \langle 5,3,2\rangle={\bf I}.$  From the description of the finite subgroups given by the classification above, we promptly  read the subgroup $K'$. More precisely the subgroup $K'$ is as follows:
\begin{itemize}
\item [(1)] finite abelian subgroup
\item[(2)] ${\bf D}_p$
\item[(3)]  $\textswab{C}_{2p}$
\item[(3')] $\textswab{C}_{p}$
\item[(4)] $\langle p,2,2\rangle={\bf D}_p$
\item[(5)] $\langle 3,3,2 \rangle={\bf  T}$
\item[(6)] $\langle 2,2,2\rangle= \ quaternionic \ group \ of \ order \  eight={\bf Q}_8$ 
\item[(7)] $\langle 4,3,2 \rangle={\bf O}=\langle X, P,Q, R \ \ | \ \ X^3=1,  \ \ P^2=Q^ 2=R^2, \ \  XPX^{-1}=Q, \ \ XQX^{-1}=PQ, \ \  RXR^{-1}=X^{-1}, \ \  RPR^{-1}=QP,  \ \    RQR^{-1}= Q^{-1}\rangle$ 
\item[(8)]  $\langle 3,3,2\rangle={\bf T}$
\item[(9)]  $\langle 5,3,2 \rangle={\bf  I}$
\end{itemize}
The automorphism $\theta$  of $K'$ which  appears in the group $K'\rtimes_{\theta}\mathbb Z$    is obtained from the conjugation by an 
element of the group $K$ of the form $(e,r)$. Here, $e$ is a generator of the cyclic  group $L$.

Based upon the classification above, we divide the cases into four subfamilies. \\
 Let $P_1$   be the family of finite subgroups  given by  (2), (5), (7)  and (9).\\
 Let $P_2$   be the family of finite subgroups  given by (1).  \\
 Let $P_3$   be the family of finite subgroups  given by  (3), (3') and (4).\\
Let $P_4$   be the family of finite subgroups  given  by  (6) and (8).
\vskip1cm

\noindent  {\bf Case $P_1$:}

For this subfamily we make use of the fact that the finite subgroup is also given by 
a presentation. We will use this presentation to describe $K'$ and $\theta$. 
The particular case of $K={\bf I}$, case $(9)$,  in fact there is no group of the form ${\bf I}\rtimes \mathbb Z$ which has the $R_{\infty}$-property. This  follows from  the fact that     ${\rm Out}({\bf I})$ is isomorphic 
to $\mathbb Z_2$, so the result follows from Proposition \ref{type2}. 

Nevertheless we will give a proof which works for all four cases in  the subfamily $P_1$.
The subgroups in question are isomorphic to the group 
 $\langle r,s,t \rangle_m$, which by \cite{Co1} is the group defined by the following presentation
$$\langle A,B,C,Z \ | \  A^r=B^s=C^t=ABC=Z^m \ \ and \  \ Z \ \  commutes \ with  \ \ A,B,C \rangle.$$
The subgroups $K'$ are $\langle p,2,2 \rangle={\bf D}_p,   \langle 3,2,2\rangle={\bf T},   \langle 4,3,2 \rangle={\bf O},   \langle 5,3,2\rangle={\bf I}$,  respectively. The action $\theta$ is the conjugation by $Z$, so we have a central extention.
Therefore for each $K$ in this subfamily, $\pi_1(U(2)/K) \cong K'\times \mathbb Z$ which does  not have  the  $R_{\infty}$-property.    

\vskip0.5cm

\noindent {\bf Case $P_2$:}

 By the description of the finite groups in (1), $K\cong \mathbb Z_m \times \mathbb Z_g$ (or $\mathbb Z_{2m} \times \mathbb Z_{g/2}$) is Abelian. Therefore  the automorphism of $K'$ is the identity
and thus $\pi_1(U(2)/K) \cong  K' \times \mathbb Z$ which does not have the $R_{\infty}$-property.

\vskip0.5cm

\noindent {\bf Case $P_3$:}   

 For the case $(3)$, the automorphism $\theta(1)$ is given by  $\theta(1)(x)=x^ {-1}$ where  $\langle x\rangle=\textswab{C}_{2m}$. Since  $\theta(1)$ has order $2$, it follows from Proposition \ref{type2} that  $\textswab{C}_{2m}\rtimes_{\theta}\mathbb Z$ does not have the $R_{\infty}$-property.   The  case of $(3' )$ is similar where  the two possibilities  for  
$\theta(1)$ are in fact the  same, i.e.   $\theta(1)(x)=x^  {-1}$.  For  the  case $(4)$, let us consider  the presentation  ${\bf D}_{2p}=\langle x,y|x^{2p}=y^ 2 \ \  yxy^ {-1}=x^ {-1}\rangle$ 
 and the corresponding presentation for 
   ${\bf D}_{p}=\langle z,y |  z^{p}=y^2 \ \  yzy^ {-1}=z^ {-1}\rangle$, as well as the inclusion
     ${\bf D}_{p} \hookrightarrow {\bf D}_{2p}$ by sending   $z\mapsto x^ 2$ and   $y\mapsto y$.  Thus, up to an inner automorphism of  $ {\bf D}_{p}$, the automorphism 
     $\theta(1)$ is given by $z\mapsto xzx^ {-1}=z$,  $y   \mapsto xyx^ {-1}=xyx^ {-1}y^ {-1}y=x^2y=zy$. The inverse   $\theta(-1)$ is given by $z\mapsto z$,  $y   \mapsto z^{-1}y$.
     Note that the  two  automorphisms $\theta(1)$ and $\theta(-1)$  are conjugated by the automorphism  which is conjugation  by $z^ {p-1}$.

\vskip0.5cm

\noindent {\bf Case $P_4$:}

 If $K'={\bf Q}_8$ in (6) then ${\rm Out}(K')\cong S_3$, the symmetric group on $3$ letters. Now $[\theta(1)]$ is of order either $2$ or $3$ in ${\rm Out}(K')$ where $[\theta(1)]$ denotes the image of $\theta(1)$ in ${\rm Out}(K')$. In $S_3$, we can conclude that $\theta(1)$ and $\theta(-1)$ are conjugate. 
If $K'={\bf T}$ in (8), then ${\rm Out}(K')\cong \mathbb Z_2$ so that $[\theta(1)]=[\theta(-1)]$. Hence,  $K' \rtimes \mathbb Z$ do not have property $R_{\infty}$. Alternatively, one can find automorphisms $\psi_1:{\bf Q}_8  \to  {\bf Q}_8$ and $\psi_2:{\bf T}  \to  {\bf T}$ such that $\psi_i\circ\psi_i=Id$ for $i=1,2$.  Then it follows that  
 $\psi_i\circ \theta(1)\circ\psi_i^{-1}=\theta(-1)$ for $i=1,2$.  Hence,  $K' \rtimes \mathbb Z$ do not have property $R_{\infty}$.

We now summarize our discussion with the following.

\begin{thm}\label{gp-U2K}
For any finite subgroup $K$ of $U(2)$, $\pi_1(U(2)/K)$ does not have property $R_{\infty}$.
\end{thm}

\subsection{The  space $U(2)/K$ and  the $R_{\infty}$-property}

As we have seen in section 2 that there are $R_{\infty}$-spaces $M$ where $\pi_1(M)$ does not have property $R_{\infty}$. Here we study the converse question. That is, we ask whether   the space $U(2)/K$ has the $R_{\infty}$-property for any finite  subgroup  $K$.   In  order to do  that, we consider the locally trivial fibration $SU(2)/K' \to U(2)/K \to S^1$ over  $S^1$ induced by the determinant map $U(2) \to S^1$. We will explicitly define a fibre-preserving self homeomorphism $\eta$ on $U(2)/K$ which in turn induces the map $r:S^1\to S^1$ of  degree  $-1$ given by  $r(z)=\bar  z$ (the complex  conjugation). Using the commutative diagram \eqref{exact}, we conclude that $R(\eta)<\infty$ so $U(2)/K$ is not an $R_{\infty}$-space.

 {\bf Case (1)} 
 
Let $\kappa \in Sp(1)$ be given by $\kappa =a+bi+cj+dk$ where $a,b,c,d \in \mathbb R$ and $\hat{\kappa}=a-bi-dj-ck$. Consider the map $\Theta : U(2) \to U(2)$ given by $(\epsilon, \kappa) \to (\epsilon^{-1}, \hat{\kappa})$. It is straightforward to check that the map $\kappa \mapsto \hat{\kappa}$ is an automorphism of $Sp(1)$ and thus $\Theta$ is an automorphism of $U(2)$. From the description of these subgroups which are finite abelian subgroups of $U(2)$, every element of $K$ is of the form \eqref{abelian-K'}. It follows that if $(\epsilon, \kappa)\in K$ then $\kappa=a+bi+0j+0k$ where $a^2+b^2=1$ so that $\hat \kappa=a-bi=\kappa^{-1}$. Thus, $(\epsilon^{-1}, \hat \kappa)=(\epsilon^{-1}, {\kappa}^{-1}) \in K$. In other words, $K$ is invariant under $\Theta$. It is easy to see that $K'$ is also invariant under $\Theta$. The quotient homeomorphism $\eta : U(2)/K \to U(2)/K$ induces the map $r:S^1\to S^1$ given by the complex conjugation $r(z)=\bar z$.

\vskip0.5cm

{\bf Cases (2) - (5), (7) - (9)}

In each of these cases, the index $[R:R_K]$ is either $1$ or $2$. Suppose $K$ is given by \eqref{finiteK} and $(\epsilon, \kappa)\in K$. Since $[R:R_K]\in \{1,2\}$, it follows that $(1,\kappa^2)\in K$. Since $(\epsilon^{-1}, \kappa^{-1})\in K$, $(\epsilon^{-1},\kappa)=(1,\kappa^2)\cdot (\epsilon^{-1}, \kappa^{-1})\in K$. Now, the map $F:U(2) \to U(2)$ given by $(\epsilon, \kappa) \to (\epsilon^{-1}, \kappa)$ is a group isomorphism which leaves $K$ and $K'$ invariant. Thus, the quotient homeomorphism $\eta$ on $U(2)/K$ induces the complex conjugation $r(z)=\overline{z}$ on $S^1$.

\vskip0.5cm

{\bf Case (6)}   

Let $\Theta :U(2)\to U(2)$ be the automorphism as in Case {\bf (1)}. The map $\kappa \mapsto \hat{\kappa}$ is an automorphism of $Sp(1)$ whose restriction to ${\bf Q}_8$ is given by $\psi: i \mapsto -i$,   $j \mapsto -k$,   $k \mapsto -j$. Note that $R={\bf T}=\{\pm1, \pm i, \pm j, \pm k, (1/2)(\pm 1  \pm i  \pm j  \pm k)\} \subset Sp(1)$. If $\kappa \in R_K={\bf Q}_8$ then $\hat \kappa \in {\bf Q}_8$. Suppose $\kappa \in {\bf T}-{\bf Q}_8$, then $\kappa=\pm a\pm bi \pm cj \pm dk$ where $|a|=|b|=|c|=|d|=1/2$. Since $|c|=|d|$, we conclude that $\hat \kappa =\kappa^{-1}$. Thus $K$ is invariant under $\Theta$. Note that in this case $K'={\bf Q}_8$ so the map $\Theta$ also keeps $K'$ invariant and it induces a self homeomorphism $\eta$ on $U(2)/K$ whose restriction on $X=SU(2)/K'$, denoted by $h$, induces on the fundamental group the automorphism $\psi$. Moreover, $\eta$ induces on $S^1$ the map $r(z)=\overline{z}$.

\vskip0.5cm

{\bf Case (3')}

Similar to Case (6), for any $\kappa =a+bi+cj+dk \in Sp(1)$, we define $\bar \kappa=a+bi-cj-dk$. Equivalently, if $\kappa =\alpha + \beta j, \alpha, \beta \in \mathbb C$, then $\bar \kappa = \alpha -\beta j$. The map $\varphi: \kappa \mapsto \bar \kappa$ is an automorphism of $Sp(1)$ so that $\Phi: U(2) \to U(2)$ given by $(\epsilon, \kappa) \mapsto (\epsilon^{-1}, \bar \kappa)$ is an automorphism. To obtain our conclusion as before, it suffices to show that the subgroups $K$ and $K'$ are invariant under $\Phi$. The group ${\bf D}_p$ of order $4p$ can be embedded in $Sp(1)$ and is generated by $u=e^{\pi i/p}$ and $w=j$ so that ${\bf D}_p=\langle u,w \mid u^{p}=w^2, wuw^{-1}=u^{-1} \rangle$. Now the subgroup $R_K=\langle u^2\rangle$. Suppose $(\epsilon, \kappa)\in K$. Since $\kappa \in R$, $\kappa=u^sw^t$. Note that 
$\varphi(u)=\varphi(e^{\pi i/p})=e^{\pi i/p}=u$ so $K'$ is invariant under $\Phi$. It follows that
\begin{equation*}
\begin{aligned}
\bar \kappa&=\varphi(u^sw^t)=u^s\varphi(w^t) \\
              &=u^s(-j)^{t}
\end{aligned}
\end{equation*}
which implies that
$$\bar \kappa=\begin{cases}
                               \kappa , \text{~for $t=0$}; \\
                               u^{s}(-j)=\kappa^{-1} , \text{~for $t=1$}; \\
                               \kappa , \text{~for $t=2$}; \\
                               u^{s}(j)=\kappa^{-1} , \text{~for $t=3$}; \\
                               \end{cases}
                               $$
It follows that $(\epsilon^{-1}, \bar \kappa)=(\epsilon^{-1}, \kappa^{-1}) \in K$ when $t$ is odd. When $t$ is even, $\kappa^2=u^{2s}\in R_K$ so $(1, \kappa^2)\in K$. In this case, 
$(\epsilon^{-1}, \bar \kappa)=(\epsilon^{-1}, \kappa)=(\epsilon^{-1}, \kappa^{-1})\cdot (1,\kappa^2)\in K$. Therefore, $K$ is invariant under $\Phi$.

\begin{rem}
In all of the cases except {\bf (6)} and {\bf (3')}, we constructed a homeomorphism $f:U(2)/K \to U(2)/K$ with $R(f)<\infty$ without specifying how we embed $K$ inside $U(2)$. For the other two cases, we use the fact that any two isomorphic finite subgroups of $Sp(1)$ are conjugate in $Sp(1)$ (see e.g. \cite[p.54]{Du} or \cite[p.68]{Co1}) so that the result  above is  independent of the embedding $K\hookrightarrow U(2)$   since $U(2)/K$ is unique.
\end{rem}

Thus, we have shown the following

\begin{thm}\label{sp-U2K}
For any finite subgroup $K$ of $U(2)$, the space $U(2)/K$ does not have property $R_{\infty}$.
\end{thm}

\begin{rem}
Although Theorem \ref{sp-U2K} implies Theorem \ref{gp-U2K}, we used different arguments which may be useful in determining the $R_{\infty}$-property of $\pi_1(U(n)/K)$ and of $U(n)/K$ respectively for $n=3,4$.
\end{rem}

\end{document}